\documentclass[12pt, reqno]{amsart}
\usepackage{amsmath, amsthm, amscd, amsfonts, amssymb, graphicx, color}
\usepackage[bookmarksnumbered, colorlinks, plainpages]{hyperref}

\makeatletter \oddsidemargin.9375in \evensidemargin \oddsidemargin
\newtheorem{theorem}{Theorem}[section]
\newtheorem{lemma}[theorem]{Lemma}

\theoremstyle{definition}
\newtheorem{definition}[theorem]{Definition}
\newtheorem{example}[theorem]{Example}

\theoremstyle{remark}
\newtheorem{remark}[theorem]{Remark}
\numberwithin{equation}{section}

\begin{document}
\setcounter{page}{1}

%%%%%%%%%%%%%%%%%%%%%%%%%%%%%%%%%%%%%%%%%%%%%%%
%% Please do not remove the following statement.
%%%%%%%%%%%%%%%%%%%%%%%%%%%%%%%%%%%%%%%%%%%%%%%
\noindent \textbf{{\footnotesize Journal of Algebraic Systems\\ Vol. XX, No XX, (201X), pp XX-XX}}\\[1.00in]
%%%%%%%%%%%%%%%%%%%%%%%%%%%%%%%%%%%%%%%%%%%%%%%

%%%%%%%%%%%%%%%%%%%%%%%%%%%%%%%%%%%%%%%%%%%%%%%%%%%%%%%%%%%%%%%%%%%%%
% Insert title of your article. Note: \title[short title]{full title}
%%%%%%%%%%%%%%%%%%%%%%%%%%%%%%%%%%%%%%%%%%%%%%%%%%%%%%%%%%%%%%%%%%%%%
\title{Ramanujan polar graphs}
%%%%%%%%%%%%%%%%%%%%%%%%%%%%%%%%%%%%%%
% Author's name must be inserted here
%%%%%%%%%%%%%%%%%%%%%%%%%%%%%%%%%%%%%%
\author[V. Smaldore]{V. Smaldore$^{*}$}

%%%%%%%%%%%%%%%%%%%%%%%%
\thanks{{\scriptsize
\hskip -0.4 true cm MSC(2020): Primary: 05C48; Secondary: 05E30,
51A50
Received: 24 February 2024, Accepted: 25 August 2024.\\
$*$Corresponding author }}

%%%%%%%%%%%%%%%%%%%%%%%%%%%%%%%%%%%%%%%%%%%
\begin{abstract}
Recently, a construction of minimal codes arising from a family of almost Ramanujan graphs was shown. Ramanujan graphs are examples of expander graphs that minimize the second-largest eigenvalue of their adjacency matrix.  We call such graphs Ramanujan, since all known non-trivial constructions imply the Ramanujan conjecture on arithmetical functions. In this paper, we prove that some families of tangent graphs of finite classical polar spaces satisfy Ramanujan's condition. If the polarity is unitary, or it is orthogonal and the quadric is over the binary field, the tangent graphs are strongly regular, and we know their spectrum. By direct computation, it is possible to show which families of tangent graphs are Ramanujan.

\end{abstract}

%%%%%%%%%%%%%%%%%%%%%%%
\maketitle
%%%%%%%%%%%%%%%%%%%%%%%
\section{Introduction}
 \textit{Expander graphs} may be defined as graphs which satisfy sorts of \textit{extremal} properties. Here, we consider the expander graphs such that the second eigenvalue is as small as possible, the so-called \textit{Ramanujan graph}. Ramanujan graphs are called this way since the known non-trivial constructions assume true the \textit{Ramanujan conjecture} on arithmetical functions, see \cite[Section 28]{Rama}.
 %\begin{con}
%  Let $\tau:\mathbb{N}\rightarrow\mathbb{Z}$ be the arithmetical function defined by the identity
%  $$\sum_{n\geq1}\tau(n)q^{n}=q\prod_{n\geq1}(1-q^{n})^{24},$$
%  where $q=e^{2\pi iz}$, with $\Im(z)>0$. Then
%  \begin{equation*}
%   |\tau(p)|\leq2p^{\frac{11}{2}}.
%  \end{equation*}
% \end{con}
 In a recent paper, \cite{ABDN}, the authors showed a construction of minimal codes arising from almost Ramanujan graphs. In this paper, we prove that some families of \textit{tangent graphs} of finite classical polar spaces are examples of Ramanujan graphs. Tangent graphs of polar spaces are graphs whose vertex set is the set of non-isotropic points, and two vertices are adjacent if and only if their connecting line is tangent to the polar space. If the polarity is unitary, or it is orthogonal and the quadric is over the binary field, the tangent graph is a strongly regular graph. We refer to \cite[Section 3.1]{brvan} for all details on tangent graphs, while more information in the unitary case is in \cite{IPS, RS1}, and in the orthogonal case in \cite{Svob} and in \cite{RS} while $n=3$. We will prove the following theorem.
 \begin{theorem}\label{main}
  The following families of graph are Ramanujan
  \begin{enumerate}
   \item $NO^{\varepsilon}(2m,2)$, for $m\geq2$;\\
   \item $NO(2m+1,2)\cong\Gamma_{Q(2m,2)}$, for $m\geq2$;\\
   \item $NU(m,4)$, for $m\geq3$.
  \end{enumerate}
 \end{theorem}

 Section \ref{sec2} give some preliminaries on extremal graphs and finite classical polar spaces, while in Section \ref{sec3} we find the constructions of tangent graphs. Finally, Section \ref{sec4} is devoted to the proof of Theorem \ref{main}.

\section{Preliminaries}\label{sec2}
\subsection{Extremal graphs}\label{subsec21}
 Let $G=(V,E)$ be a (simple, non-directed) graph of $n$ vertices. The \textit{adjacency matrix} $A$ of $G$ is a $n\times n$ $(0,1)$-symmetric matrix whose entries are indexed by the vertices of $G$ and where $A(i,j)=1$ if and only if $v_{i}$ and $v_{j}$ are adjacent vertices. The \textit{spectrum} of a graph $G$ is the multiset of the eigenvalues of its adjacency matrix, counted with their respective multiplicities.

 If $G$ is connected and $d$-regular, then $A$ has $(1,\ldots,1)^{T}$ as eigenvector, with relative eigenvalue $d$ of multiplicity 1. We denote the eigenvalues of $A$ by $d=\lambda_{1}>\lambda_{2}>\ldots>\lambda_{j}$, with relative multiplicities $1=m_{1},m_{2},\ldots,m_{j}$, and we write the spectrum as $(d,\lambda_{2}^{m_{2}},\ldots,\lambda_{j}^{m_{j}})$. A graph $G$ is called an $(n,d,\lambda)$-graph if it is a $d$-regular graph on $n$ vertices with $\lambda_2=max_{i>1}|\lambda_{i}|=\lambda$. A classical tool for studying $(n,d,\lambda)$-graphs is the following lemma.

 \begin{lemma}[\textbf{Expander-mixing lemma}]
  Let $G=(V,E)$ be an $(n,d,\lambda)$-graph and $S,T$ be two vertex subsets of $V$. Denote by $e(S,T)$ the number of edges in $E$ connecting a vertex of $S$ and a vertex of $T$. Then,
  \begin{equation*}
   \Big|e(S,T)-\frac{d|S||T|}{n}\Big|\leq\lambda\sqrt{|S||T|\Big(1-\frac{|S|}{n}\Big)\Big(1-\frac{|T|}{n}\Big)}.
  \end{equation*}
 \end{lemma}

 \normalsize
 Since the second term depends on $\lambda$, it is natural to try to minimize this term. A lower bound for how small $\lambda$ can be was given in \cite{N}.
 \begin{theorem}
  In a $(n,d,\lambda)$-graph $G$ it asymptotically holds:
  \begin{equation}\label{Rama}
   \lambda\geq2\sqrt{d-1}-o(1).
  \end{equation}
 \end{theorem}
 Ramanjuan graphs minimize asymptotically the value of $\lambda$, see \cite{HLW}.
 \begin{definition}
  A $(n,d,\lambda)$-graph is called \textit{Ramanujan graph} if
  $$\lambda\leq2\sqrt{d-1}.$$
 \end{definition}
 As the smallest eigenvalue is $-d$ if and only if the graph is bipartite, we may refine the definition in order to allow \textit{bipartite Ramanujan graphs}, not considering the eigenvalue $-d$.
  \begin{definition}
  A $(n,d,\lambda)$-bipartite graph is called \textit{bipartite Ramanujan graph} if
  \begin{equation*}
   max_{i:|\lambda_{i}|<d}|\lambda_{i}|\leq2\sqrt{d-1}.
  \end{equation*}
 \end{definition}

 We list above some known examples of Ramanujan graphs:
 \begin{itemize}
  \item The complete graph $K_{d+1}$ has spectrum $(d,-1^{d})$, and then, while $d>1$, it is a Ramanujan graph.
  \item The complete bipartite graph $K_{d,d}$ has spectrum $(d,0^{2d-2},-d)$, and then, while $d>1$, it is a bipartite Ramanujan graph.
  \item Payley graphs of order $q\equiv1\pmod5$ have regularity $\frac{q-1}{2}$ and the other eigenvalues are $\frac{-1\pm\sqrt{q}}{2}$, and then they are Ramanujan graphs $(q\neq1)$.
  \item Lubotzky-Phillips-Sarnak construction, see \cite{LPS}, gives rise to a family of $(p+1)$-regular graphs $X^{p,r}$, where $p\equiv1\pmod4$ is a prime number. $X^{p,r}$ is the Cayley graph with $p+1$ generators associated with the solutions of the equation $a_0^2+a_1^2+a_2^2+a_3^2=p$, $a_0>0$ odd and $a_1, a_2, a_3$ even, defined on $PSL(2,\mathbb{Z}/p\mathbb{Z})$, or respectively on $PGL(2,\mathbb{Z}/p\mathbb{Z})$, according if $p$ is a quadratic residuo modulo $r$ or not. Then $X^{p,r}$ is a Ramanujan graph. The construction was later generalized by Morgenstern in \cite{M}, for all prime powers $p$, taking a prime $r\equiv1\pmod4$, $r\neq p$.
 \end{itemize}

 \subsection{Finite classical polar spaces}\label{subsec22}
  Let $PG(N,q)$ denote the $N$-dimensional projective space over the finite field $\mathbb{F}_{q}$. A non-degenerate sesquilinear or non-singular quadratic form on the underlying $(N+1)$-dimensional vector space induces a geometry embedded in $PG(N,q)$, which will be called a \textit{finite classical polar space}, or simply polar space. Its elements are the totally isotropic, respectively totally singular, subspaces with relation to the sesquilinear, respectively quadratic, form. The subspaces of maximal dimension contained in a polar space are called \textit{generators}, and the projective dimension of a generator is $m-1$. The vectorial dimension $m$ of a generator is called \textit{rank} of the polar space. Based on the classification of non-degenerate sesquilinear, respectively non-singular, quadratic forms over finite fields, one distinguishes between 3 families of polar spaces: orthogonal polar spaces, symplectic polar spaces, and Hermitian polar spaces. Symplectic polar spaces live only in projective spaces of odd dimensions, and they correspond to alternating sesquilinear forms. Orthogonal polar spaces correspond to symmetric forms, and they fall apart in 3 subfamilies, i.e. hyperbolic quadrics, elliptic quadrics (both in odd projective dimension), and parabolic quadrics (in even projective dimension). We distinguish between elliptic and hyperbolic quadrics by their rank. Hermitian polar spaces correspond to Hermitian forms, and we distinguish between odd and even projective dimension cases as well. Note that Hermitian polar spaces are only defined over a field of square order. The type of the polar space $e$ allows us to identify different families of polar spaces with the same rank.
  \begin{table}
\begin{center}
\begin{tabular}{|c|c|c|}
\hline
\textbf{Polar space} & \textbf{Projective dimension} & $\mathbf{e}$ \\
\hline
Elliptic quadric $Q^-(2m+1,q)$ & $2m+1$ & 2 \\
\hline
Hyperbolic quadric $Q^+(2m-1,q)$ & $2m-1$ & 0 \\
\hline
Parabolic quadric $Q(2m,q)$ & $2m$ & 1 \\
\hline
Symplectic space $W(2m-1,q)$ & $2m-1$ & 1 \\
\hline
Hermitian space $H(2m,q)$ & $2m$ & $\frac{3}{2}$ \\
\hline
Hermitian space $H(2m-1,q)$ & $2m-1$ & $\frac{1}{2}$ \\
\hline
\end{tabular}
\caption{$\mathcal{P}_{m,e}$ polar space of rank $m \geq 1$}\label{tab:polarspaces}
\end{center}
\end{table}
\normalsize
The notation $\mathcal{P}_{m,e}$ stands for a polar space of rank $m$ and type $e$, listed in Table~\ref{tab:polarspaces}.

\section{Tangent graphs}\label{sec3}
In this section, we will define the parameters of tangent graphs w.r.t. a polar space. Roughly speaking, those graphs have, as vertex set, the set of non-isotropic points and two vertices are adjacent if and only if the corresponding points lie on a line that is tangent to the polar space. See \cite[Section 3.1]{brvan} for all details. More information can be found also in \cite{Svob,RS1,RS}.
 \subsection{$NO^{\varepsilon}(2m,2)$}\label{subsec31}
  Let $PG(2m-1,2)$ be a projective space of odd projective dimension over the binary field, and let $Q^{\varepsilon}(2m-1,2)$ be a non-degenerate quadric, hyperbolic or elliptic when $\varepsilon=+,-$, respectively. Then $NO^{\varepsilon}(2m,2)$ is the strongly regular graph with vertex set $PG(2m-1,2)\setminus Q^{\varepsilon}(2m-1,2)$, where two vertices are adjacent if and only if the points are orthogonal. The graph has $n=2^{2m-1}-\varepsilon2^{m-1}$ vertices, it is $d$-regular with $d=2^{2m-2}-1,$ and since the graph is strongly regular it has exactly other two eigenvalues $\lambda_{2}=\varepsilon2^{m-2}-1$ and $\lambda_{3}=-\varepsilon2^{m-1}-1$.

  \subsection{$NO(2m+1,2)$}\label{subsec32}
  Let $PG(2m,2)$ be a projective space of even projective dimension over the binary field, let $Q(2m,2)$ be a non-degenerate parabolic quadric, and let $N$ be the nucleus of $Q(2m,2)$, i.e. the intersecting point of all its tangent lines. Then $NO(2m+1,2)$ is the strongly regular graph with vertex set $PG(2m,2)\setminus (Q(2m,2)\cup\{N\})$, where two vertices are adjacent if and only if the points are orthogonal. The graph has $n=2^{2m}-1$ vertices, is $d$-regular with $d=2^{2m-1}-2,$ and since the graph is strongly regular it has exactly other two eigenvalues $\lambda_{2,3}=-1\pm2^{m-1}$. Note that the graph is isomorphic to both $\Gamma_{Q(2m,2)}$ and $\Gamma_{W(2m-1,2)}$, the collinearity graphs of the parabolic quadric $Q(2m,2)$ and the symplectic space $W(2m-1,2)$, where the collinearity graph $\Gamma_{\mathcal{P}_{m,e}}$ of the polar space $\mathcal{P}_{m,e}$ is the graph whose vertex set is $\mathcal{P}_{m,e}$, and two vertices are adjacent if they span an isotropic line, see \cite{brvan}.

  \subsection{$NU(m,q^2)$}\label{subsec33}
  Let $PG(m-1,q^2)$ be a projective space over a field of square order, and let $H(m-1,q^2)$ be a non-degenerate Hermitian variety. Then $NU(m,q^2)$ is the strongly regular graph with vertex set $PG(m-1,q^2)\setminus H(m-1,q^2)$, where two vertices are adjacent if and only if they are joined by a line tangent to $H(m-1,q^2)$. Define $\varepsilon=(-1)^m$. The graph has $n=\frac{q^{m-1}(q^{m}-\varepsilon)}{q+1}$ vertices, is $d$-regular with $d=(q^{m-1}+\varepsilon)(q^{m-2}-\varepsilon)$ and since it is strongly regular it has exactly other two eigenvalues $\lambda_{2}=\varepsilon q^{m-2}-1$ and $\lambda_{3}=-\varepsilon q^{m-3}(q^2-q-1)-1$. Here, we focus on the case $q=2$. The graph $NU(m,4)$ has $2^{2m-1}-\varepsilon2^{m-1}$ vertices, is $d$-regular with $d=(2^{m-1}+\varepsilon)(2^{m-2}-\varepsilon)$ and since the graph is strongly regular it has exactly other two eigenvalues $\lambda_{2}=\varepsilon 2^{m-2}-1$ and $\lambda_{3}=-\varepsilon 2^{m-3}-1$.

 \section{Proof of Theorem \ref{main}}\label{sec4}
  We start by considering graphs $NO^{\varepsilon}(2m,2)$. By a straightforward calculation we get the following results:
  \begin{lemma}\label{NOlem}
   \begin{itemize}
    \item $\lambda(NO^{+}(2m,2))=2^{m-1}+1$.
    \item $\lambda(NO^{-}(2m,2))=2^{m-1}-1$ for $m\geq3$, while $\lambda(NO^{-}(4,2))=2$.
   \end{itemize}
  \end{lemma}

  \begin{proof}
   Compare the values of $\lambda_{2,3}$ in Subsection \ref{subsec31}
  \end{proof}

  \begin{example}
   Graphs $NO^{\varepsilon}(2m,2)$ with $m=1,2$ may be characterized as follows:
   \begin{enumerate}
    \item $NO^{+}(2,2)\cong K_{1}$.\\
    \item $NO^{-}(2,2)\cong 3K_{1}$.\\
    \item $NO^{+}(4,2)\cong K_{3,3}$.\\
    \item $NO^{-}(4,2)\cong P$, where $P$ is the \textit{Petersen graph}.\\
   \end{enumerate}
  \end{example}

  \begin{figure}\label{f1}
   \centering
   \includegraphics[scale=0.7]{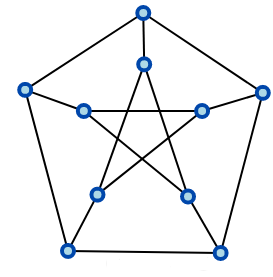}
   \caption{Petersen graph $P\cong NO^{-}(4,2)$}
  \end{figure}

  \begin{proof}
   We give some explicit calculations in the case $m=2$. A hyperbolic quadric $Q^{+}(3,2)$ has 9 points on two reguli of 6 lines. Each of the 6 points $P_{i}$ on $PG(3,2)\setminus Q^{+}(3,2)$ lies on 3 tangent lines meeting the quadric in 3 points of a conic $\mathcal{C}$. Out of $Q^{+}(3,2)$ we find two skew external \textit{nuclei lines}, representing the bipartition of the graph $NO^{+}(4,2)\cong K_{3,3}$.\\

   \textit{Petersen graph} is a 3-regular strongly regular graph with 10 vertices, any two adjacent vertices have no common neighbours and any two non-adjacent vertices have exactly one common neighbour. An elliptic quadric $Q^{-}(3,2)$ is made of 5 points, no three on a line, i.e. it is a 5-ovoid of $PG(3,2)$. Fix an external line $\ell$, then two points $P_1$ and $P_2$ on $\ell$ represent two non-adjacent vertices, $P_1$ lies on 3 tangent lines meeting the quadric in 3 points of a conic $\mathcal{C}_1$, and analogously $P_2$ lies on 3 tangent lines meeting the quadric in 3 points of a conic $\mathcal{C}_2$. Then $\mathcal{C}_1$ and $\mathcal{C}_2$ meet themselves in exactly one point. Therefore, two non-adjacent vertices of $NO^{-}(4,2)$ have exactly one common neighbour. Now fix a line $r$ tangent to $Q^{-}(3,2)$ at $T$. The other two points $T_1$ and $T_2$ on $r$ represent adjacent vertices, $T_1$ lies on 3 tangent lines meeting the quadric in 3 points of a conic $\mathcal{C}_A$, and analogously $T_2$ lies on 3 tangent lines meeting the quadric in 3 points of a conic $\mathcal{C}_B$. Then $\mathcal{C}_A$ and $\mathcal{C}_B$ meet themselves exactly in the point $T$. Therefore, two adjacent vertices of $NO^{-}(4,2)$ have no common neighbour.
  \end{proof}

  Now we consider the Hermitian space $H(m-1,4)$.
  \begin{example}
   Graphs $NU(m,4)$ with $m=2,3$ may be characterized as follows:
   \begin{enumerate}
    \item $NU(2,4)\cong2K_1$, and it is trivially non-Ramanujan.\\
    \item $NU(3,4)\cong\overline{4K_3}$, see \cite[Subsection 4.2]{RS1}.
   \end{enumerate}
   \end{example}
  \begin{figure}\label{f2}
   \centering
   \includegraphics[scale=0.8]{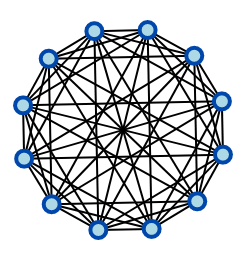}
   \caption{$NU(3,4)\cong\overline{4K_3}$}
  \end{figure}
    Some direct calculation gives the following result.
  \begin{lemma}\label{NUlem}
  \begin{itemize}
   \item If $m>2$ is even, $\lambda(NU(m,4))=2^{m-2}-1$.
  \item If $m$ is odd, $\lambda(NU(m,4))=2^{m-2}+1$.
  \end{itemize}
  \end{lemma}

  \begin{proof}
   Compare the values of $\lambda_{2,3}$ in Subsection \ref{subsec33}
  \end{proof}

We now have all the ingredients for the proof of the main theorem.
\begin{proof}[Proof of Theorem \ref{main}]
We give the proof of the theorem by comparing the smallest eigenvalues as in Lemmas \ref{NOlem} and \ref{NUlem}.
  \begin{enumerate}
   \item $NO^{\varepsilon}(2,2)$ is trivially non-Ramanujan since $d=0$, while $NO^{+}(4,2)$ is a complete bipartite graph, and then it is bipartite Ramanujan, see Subsection \ref{subsec21}. Moreover, the Petersen graph $NO^{-}(4,2)$ is Ramanujan since $d=3$, $\lambda=2$ and $2\leq2\sqrt{2}$. Now we analyze the inequality arising from Equation \eqref{Rama}:
  $$2^{m-1}+\varepsilon\leq2\sqrt{2^{2m-2}-2},$$
  $$(2^{m-1}+\varepsilon)^2\leq4(2^{2m-2}-2),$$
  $$2^{2m-2}+\varepsilon2\cdot2^{m-1}+1\leq4\cdot2^{2m-2}-8.$$
  By substituting $x=2^{m-1}$ we get
  $$3x^{2}-\varepsilon2x-9\geq0.$$
  In $"+"$ case we get $2^{m-1}=x\geq2,097\ldots$, then $NO^{+}(2m,2)$ is Ramanujan for $m\geq3$, in $"-"$ case we get $2^{m-1}=x\geq1,431\ldots$, then $NO^{-}(2m,2)$ is Ramanujan when $m\geq2$.\\
   \item In $m=1$ case, $\Gamma_{\mathcal{C}}\cong 3K_{1}$, since the conic $\mathcal{C}=Q(2,2)$ on $PG(2,2)$ consists on a 3-arc and it does not contain lines. Then the graph is trivially non-Ramanujan. While $m\geq2$, $\lambda(\Gamma_{Q(2m,2)})=2^{m-1}+1$. Then the inequality arising from Equation \eqref{Rama} gives:
  $$2^{m-1}+1\leq2\sqrt{2(2^{2m-2}-1)-1},$$
  $$(2^{m-1}+1)^2\leq8(2^{2m-2}-1)-4,$$
  $$2^{2m-2}+2\cdot2^{m-1}+1\leq8\cdot2^{2m-2}-12,$$
  $$7\cdot2^{2m-2}-2\cdot2^{m-1}-13\geq0.$$
  By substituting $x=2^{m-1}$ we get
  $$7x^{2}-2x-13\geq0.$$
  We get $2^{m-1}=x\geq1,513\ldots$, then $\Gamma_{Q(2m,2)}$ is Ramanujan when $m\geq2$.\\
   \item[3.1] In the $m$ even case, the inequality arising from Equation \eqref{Rama} gives:
  $$2^{m-2}-1\leq2\sqrt{(2^{m-1}+1)(2^{m-2}-1)-1},$$
  $$(2^{m-2}-1)^2\leq4(2^{2m-3}-2^{m-1}+2^{m-2}-2),$$
  $$2^{2m-4}-2\cdot2^{m-2}+1\leq4\cdot2^{2m-3}-4\cdot2^{m-1}+4\cdot2^{m-2}-8,$$
  $$2^{2m-4}-2\cdot2^{m-2}+1\leq8\cdot2^{2m-4}-8\cdot2^{m-2}+4\cdot2^{m-2}-8,$$
  $$7\cdot2^{2m-4}-2\cdot2^{m-2}-9\geq0.$$
  By substituting $x=2^{m-2}$ we get
  $$7x^{2}-2x-9\geq0.$$
  We get $2^{m-2}=x\geq1,286\ldots$, then $NU(m,4)$, $m$ even, is Ramanujan when $m\geq4$.\\
 \item[3.2] In the $m$ odd case, the inequality arising from Equation \eqref{Rama} gives:
  $$2^{m-2}+1\leq2\sqrt{(2^{m-1}-1)(2^{m-2}+1)-1},$$
  $$(2^{m-2}+1)^2\leq4(2^{2m-3}+2^{m-1}-2^{m-2}-2),$$
  $$2^{2m-4}+2\cdot2^{m-2}+1\leq4\cdot2^{2m-3}+4\cdot2^{m-1}-4\cdot2^{m-2}-8,$$
  $$2^{2m-4}+2\cdot2^{m-2}+1\leq8\cdot2^{2m-4}+8\cdot2^{m-2}-4\cdot2^{m-2}-8,$$
  $$7\cdot2^{2m-4}+2\cdot2^{m-2}-9\geq0.$$
  By substituting $x=2^{m-2}$ we get
  $$7x^{2}+2x-9\geq0.$$
  Since $2^{m-2}=x\geq1$, $NU(m,4)$, $m$ odd, is Ramanujan when $m\geq3$.
  \end{enumerate}
\end{proof}

 \begin{remark}
  We end the section with a few words about the case $q>2$. As $q$ grows, the tangent unitary graphs are non-Ramanujan. In fact, we get $\lambda(NU(m,q^2))=q^{m-3}(q^2-q-1)+\varepsilon$ and the inequality arising from Equation \eqref{Rama} gives:
  $$q^{m-3}(q^2-q-1)+\varepsilon\leq2\sqrt{(q^{m-1}+\varepsilon)(q^{m-2}-\varepsilon)-1},$$
  $$(q^{m-3}(q^2-q-1)+\varepsilon)^2\leq4(q^{2m-3}-\varepsilon q^{m-1}+\varepsilon q^{m-2}-2).$$
 We get
  $$q^{2m-2}-6q^{2m-3}-q^{2m-4}+2q^{2m-5}+q^{2m-6}+\varepsilon6q^{m-1}-\varepsilon6q^{m-2}-\varepsilon2q^{m-3}+8\leq0,$$
  that holds asymptotically false, as $q\rightarrow\infty$.
  \end{remark}

  \section{Conclusion}
  In \cite{ABDN}, \textit{almost Ramanujan graphs} found applications in strong blocking sets, which in turn are applied to minimal linear block codes and perfect hash families. Roughly speaking, it is crucial the (extremal) property called \textit{integrity}, which relies on the size of the maximum connected component of a subgraph, and whose value is strictly connected to the largest eigenvalue of a $(n,d,\lambda)$-graph. Moreover in \cite{SS}, the authors provided a construction of a family of \emph{LDPC-codes}, i.e. codes with a low-density parity-check matrix, which arise from bipartite expander graphs. Further developments may find analogous applications of Ramanujan polar graphs in coding theory.

%------------------------------------------------------------------------------------%

\vskip 0.4 true cm

\begin{center}{\textbf{Acknowledgments}}
\end{center}
 The research was partially supported by the Italian National Group for Algebraic and Geometric Structures and their Applications (GNSAGA - INdAM) and by the INdAM - GNSAGA Project \emph{Tensors over finite fields and their applications}, number E53C23001670001. Both figures belong to the \textit{House of Graphs} database, \cite{HoG}. \\ \\
\vskip 0.4 true cm

%------------------------------------------------------------------------------------%
%---------------------------------------------------------------------------------------%
%%%%%%%%%%%%%%%%%%%%%%%%%%%%%%%%%%%%%%%%%%%
% References
%%%%%%%%%%%%%%%%%%%%%%%%%%%%%%%%%%%%%%%%%%%
\bibliographystyle{amsplain}
%%%%%%%%%%%%%%%%%%%%%%%%%%%%%%%%%%%%%%%%%%%
% Please cite your relevant papers but at most total 5 papers/books.
%%%%%%%%%%%%%%%%%%%%%%%%%%%%%%%%%%%%%%%%%%%

\bigskip
\bigskip

%{\bf Received: Month xx, 200x}

{\footnotesize{\bf Valentino Smaldore}\; \\ {Dipartimento di Tecnica e Gestione dei Sistemi Industriali}, {Universit\`{a} degli Studi di Padova, Stradella S. Nicola 3, 36100} {Vicenza, Italy.}\\
{\tt Email: valentino.smaldore@unipd.it}\\

\end{document}